\documentclass{amsart}
\usepackage{amssymb}
\usepackage{amscd}
\usepackage[latin1]{inputenc}
\usepackage{amsmath}
\usepackage{amsfonts}
\usepackage{latexsym}
\usepackage{verbatim}
\usepackage[pdftex]{graphicx}
\usepackage{epsfig}
\usepackage{color}
\newtheorem{theorem}{Theorem}[section]
\newtheorem{corollary}[theorem]{Corollary}
\newtheorem{lemma}[theorem]{Lemma}

\theoremstyle{definition}
\newtheorem{definition}[theorem]{Definition}
\newtheorem{example}[theorem]{Example}
\newtheorem{remark}[theorem]{Remark}

\numberwithin{equation}{section}

\begin{document}
\title[Follower, predecessor, and extender entropies]{Follower, predecessor, and extender entropies}

\begin{abstract}

Using the follower/predecessor/extender set sequences defined in \cite{french1}, we define quantities which we call the follower/predecessor/extender entropies, which can be associated to any shift space. We analyze the behavior of these quantities under conjugacies and factor maps, most notably showing that extender entropy is a conjugacy invariant and that having follower entropy zero is a conjugacy invariant. We give some applications, including examples of shift spaces with equal entropy which can be distinguished by extender entropy, and examples of shift spaces which can be shown to not be isomorphic to their inverse by using follower/predecessor entropy.

\end{abstract}

\date{}
\author{Thomas French and Ronnie Pavlov}
\address{Thomas French and Ronnie Pavlov\\
Department of Mathematics\\
University of Denver\\
2390 S. York St.\\
Denver, CO 80208}
\email{Thomas.French@du.edu}
\email{rpavlov@du.edu}
\urladdr{http://www.math.du.edu/$\sim$rpavlov}
\thanks{The second author gratefully acknowledges the support of NSF grant DMS-1500685.}

%\thanks{}
%\keywords{$\mathbb{Z}^d$; shift of finite type; sofic; multidimensional}
%\renewcommand{\subjclassname}{MSC 2010}
%below are the definitions for some subject classification numbers
%22D40 Ergodic theory on groups 
%37A05 Measure-preserving transformations 
%37A15 General groups of measure preserving transformations 
%37A35 Entropy and other invariants, isomorphism, classification 
%37B10 symbolic dynamics
%37B40 topological entropy
%37B50 multi-dimensional shifts of finite type, tiling systems 
%37C40 smooth ergodic theory, invariant measures
%37C45 dimension theory of dynamical systems
%37C85 Dynamics of group actions other than Z and R, and foliations 
%37C99 smooth dynamical systems, general theory
%37D35 thermodynamic formalism, variational principles, equilibrium states
\maketitle
%\tableofcontents

\section{Introduction}
\label{intro}

Topological entropy is one of the most well-studied and useful invariants of topological dynamical systems. In the specific case of (one-dimensional) shift spaces, topological entropy is precisely the exponential growth of the number of words in the language of the shift space, as the length $n$ grows. 

In \cite{french1}, the first author used the classical follower/predecessor sets and the extender sets introduced in \cite{KM} to create analogues of the word complexity function, called the follower set sequence, predecessor set sequence, and extender set sequence. In this work, we use these sequences to create analogues of entropy, which we call follower/predecessor/extender entropy. 

These turn out to have some interesting properties in terms of behavior under conjugacy/factor maps.
In particular, extender entropy is a conjugacy invariant (Theorem~\ref{extconjinv}), but it need not decrease under factors (Theorem~\ref{entfact}). Follower/predecessor entropy is not a conjugacy invariant (Theorem~\ref{follnotconjinv}), but having zero follower/predecessor entropy is a conjugacy invariant (Corollary~\ref{0follentinv}).

The properties we prove for extender entropy are reminiscent of the (left/right) constraint entropy defined in \cite{buzzi}, which is also a conjugacy invariant which may increase under factor maps. 
We define constraint entropy in Section~\ref{defns}, and there compare and contrast it with extender entropy.

Finally, we give two applications of our results. The first is a way to show that a shift space is not conjugate to its inverse. Specifically, our results show that any shift space with zero follower entropy and positive predecessor entropy cannot be conjugate to its inverse. We then use results from \cite{french2} and \cite{schmeling} to show that for all $\beta$ outside a meager set of Lebesgue measure zero, the $\beta$-shift $X_\beta$ is not conjugate to its inverse (Theorem~\ref{notinv}). 

Our second application is that extender entropy can theoretically help to distinguish between shift spaces with the same topological entropy; in particular, we prove that the only restrictions on extender and topological entropy is that the first must be less than or equal to the second (Theorem~\ref{realization}).

\section{Definitions and preliminaries}

\label{defns}

\begin{definition}
A \textbf{topological dynamical system} is a pair $(X, T)$ where $X$ is a compact metric space and $T: X \rightarrow X$ is continuous. If $T$ is also a homeomorphism, then $(X,T)$ is \textbf{invertible}.
\end{definition}

In this paper, the topological dynamical systems considered will be symbolically defined. Let $A$ denote a finite set, which we will refer to as our alphabet. 

\begin{definition}
A \textbf{word} over $A$ is a member of $A^n$ for some $n \in \mathbb{N}$. We denote the length of a word $w$ by $|w|$ and the set of all words on $A$ by $A^*$.
\end{definition}

\begin{definition} For any words $v \in A^n$ and $w \in A^m$, we define the \textbf{concatenation} $vw$ to be the pattern in $A^{n+m}$ whose first $n$ letters are the letters forming $v$ and whose next $m$ letters are the letters forming $w$.
\end{definition}

\begin{definition} A word $w$ is a \textbf{prefix} of a right-infinite sequence $z$ if the first $|w|$-many letters of $z$ are the letters forming $w$. We denote the n-letter prefix of a sequence $z$ by $(z)_n$.
\end{definition}

\begin{definition}
A \textbf{shift space} is a set $X \subset A^{G}$ ($G = \mathbb{Z}$ or $\mathbb{N}$) which is closed and invariant under the left shift map $\sigma$ defined by $(\sigma x)(n) = x(n+1)$ for $n \in G$. A shift space is called \textbf{two-sided} if $G = \mathbb{Z}$ and \textbf{one-sided} if $G = \mathbb{N}$. All shift spaces we consider are two-sided unless explicitly specified otherwise.
\end{definition}

For any shift space $X$, $(X, \sigma)$ is a topological dynamical system, and it is invertible if $X$ is a two-sided shift space. In that case, $(X, \sigma^{-1})$ is also a topological dynamical system.

\begin{definition} 
The \textbf{language} of a shift space $X$, denoted by $L(X)$, is the set of all words which appear in points of $X$. For any finite $n \in \mathbb{N}$, $L_n(X) := L(X) \cap A^n$, the set of words in the language of $X$ with length $n$. The \textbf{complexity sequence} of a shift space $X$ is $\{|L_n(X)|\}_{n \in \mathbb{N}}$. That is, the complexity sequence is the sequence which records the number of words of length $n$ appearing in some point of $X$ for every length $n$.
\end{definition}

\begin{definition}
For any shift space $X$ over the alphabet $A$, and any word $w$ in the language of $X$, we define the \textbf{follower set of $w$ in $X$}, $F_X(w)$, to be the set of all finite words $u \in L(X)$ such that the word $wu$ occurs in some point of $X$. The \textbf{predecessor set of $w$ in $X$}, $P_X(w)$, is defined to be the set of all finite words $s \in L(X)$ such that the word $sw$ occurs in some point of $X$. In some works, the follower and predecessor sets have been defined to be the set of all one-sided infinite sequences (in $A^{\mathbb{N}}$ or $A^{-\mathbb{N}}$ for followers and predecessors, respectively) which may follow/precede $w$. This definition is equivalent for followers, and in the case of a two-sided shift, for predecessors as well. For a one-sided shift, of course, no infinite sequence may precede a word $w$. The results of this paper will apply for either definition in any case where that definition makes sense.
\end{definition}

\begin{definition}
For any shift space $X$ over the alphabet $A$, and any word $w$ in the language of $X$, we define the \textbf{extender set of w in $X$}, $E_X(w)$, to be the set of all pairs $(s, u)$ where $s, u \in L(X)$ and the word $swu$ occurs in some point of $X$. Again, a definition replacing finite words with infinite sequences is equivalent in the two-sided case.
\end{definition}

\begin{remark}\label{proj}
For any word $w \in L(X)$, define the surjective projection function $f_w: E_X(w) \rightarrow F_X(w)$ by $f(s,u) = u$. %Such a function sends the extender set of $w$ onto the follower set of $w$. 
Any two words $w, v$ with the same extender set would have the property then that $f_w(E_X(w))= f_v(E_X(v))$, that is, that $w$ and $v$ have the same follower set. Similarly, words which have the same extender set also have the same predecessor set.
\end{remark}

\begin{definition}
For any positive integer $n$, define the sets $F_X(n) = \{F_X(w) \ | \ w \in L_{n}(X)\}$, $P_X(n) = \{P_X(w) \ | \ w \in L_n(X)\}$, and $E_X(n) = \{E_X(w) \ | \ w \in L_n(X)\}$.%, so that $|P_X(n)|$ is the number of distinct predecessor sets of words of length $n$ in $X$ and $|E_X(n)|$ is the number of distinct extender sets of words of length $n$ in $X$.
\end{definition}

\begin{definition}
Given a shift space $X$, the \textbf{follower set sequence of $X$} is the sequence $\{|F_X(n)|\}_{n \in \mathbb{N}}$, the \textbf{predecessor set sequence of $X$} is the sequence $\{|P_X(n)|\}_{n \in \mathbb{N}}$, and the \textbf{extender set sequence of $X$} is the sequence $\{|E_X(n)|\}_{n \in \mathbb{N}}$. These sequences measure the number of distinct follower/predecessor/extender sets of words of length $n$ in $X$.
\end{definition}

By Remark~\ref{proj}, for any $X$ and $n$, 
\begin{equation}\label{count}
|E_X(n)| \geq |P_X(n)| \textrm{ and } |E_X(n)| \geq |F_X(n)|.
\end{equation}

\begin{example}
The \textbf{full shift} on the alphabet $A$ is just the shift space $X = A^{\mathbb{Z}}$. Then any word $w \in L(X)$ may be followed legally by any word $v \in L(X)$, and thus the follower sets of all words are the same. Hence there is only one follower set in a full shift. Similarly, there is only one predecessor and one extender set in a full shift. Then $\{|F_X(n)|\}_{n \in \mathbb{N}} = \{|P_X(n)|\}_{n \in \mathbb{N}} = \{|E_X(n)|\}_{n \in \mathbb{N}} = \{1, 1, 1, ...\}$.
\end{example}

\begin{example}
The \textbf{even shift} is the sofic shift space $X$ with alphabet $\{0,1\}$ defined by forbidding runs of $0$ symbols of odd length between two nearest $1$ symbols. It is a simple exercise to show that the even shift has exactly three follower sets: $F_X(0)$, $F_X(1)$, and $F_X(10)$. It also has six extender sets: $E_X(0)$, $E_X(1)$, $E_X(00)$, $E_X(01)$, $E_X(10)$, and $E_X(010)$. The follower set sequence of the even shift is $\{|F_X(n)|\}_{n \in \mathbb{N}} = \{2, 3, 3, 3, ...\}$ (the predecessor set sequence is the same by symmetry of $X$) and the extender set sequence is $\{|E_X(n)|\}_{n \in \mathbb{N}} = \{2, 5, 6, 6, ...\}$.%It is easy to verify that for any word $w$ in the language of the even shift, the follower set of $w$ is identical to the follower set of one of these three words.
\end{example}

\begin{example}\label{CFS}
The \textbf{context-free shift} $C$ is the shift space with alphabet $\{a, b, c\}$ consisting of all $x$ where any two nearest $c$ symbols must have a word of the form $a^nb^n$ between them. %The shift being closed means that points beginning with an infinite sequence of $b$'s or ending with an infinite sequence of $a$'s are legal, as well as the points $a^\infty b^\infty$, $a^\infty$, and $b^\infty$. 
The follower set of a word in the context-free shift depends only on its final letter, the location of the final $c$ appearing in the word (if any), and the location after that $c$ where the first $b$ appears (if any). Thus we have the upper bound $|F_C(n)| < 3n^2$. Similarly, the extender set of a word in the context-free shift depends only on its first and last letters, the locations of the first and last $c$ in the word (if any), as well of the locations of the last $a$ before the first $c$, and the first $b$ after the final $c$ (if any), yielding the upper bound $|E_C(n)| < 9n^4$.
\end{example}

\begin{definition}
A \textbf{factor map} from a topological dynamical system $(X,T)$ to a topological dynamical system 
$(Y,S)$ is a surjective map $\varphi: X \rightarrow Y$ where for any $x \in X$, $S\varphi(x) =  \varphi(Tx)$. If $\varphi$ is a bijection, then it is called a \textbf{conjugacy} and we say that $(X, T)$ and $(Y, S)$ are \textbf{conjugate}.
\end{definition}

The well-known Curtis-Lyndon-Hedlund theorem states that a factor map $\phi$ from a two-sided shift space $(X, \sigma)$ to another two-sided shift space $(Y, \sigma)$ must be a so-called sliding block code, i.e. there exists $r$ (called the \textbf{radius} of $\phi$) so that $(\phi(x))(i)$ depends only on $x(i-r) \ldots x(i + r)$ for all $x \in X$ and $i \in \mathbb{Z}$. A factor map with radius $0$ is called \textbf{$1$-block}, because it is induced by a map between the alphabets of the shifts.

\begin{definition}
Given a shift space $X$, the \textbf{topological entropy of $X$} is given by $$h(X) = \lim_{n \rightarrow \infty} \frac{1}{n} \log{|L_n(X)|}.$$ It is well-known that topological entropy is a conjugacy invariant.
\end{definition}

\begin{definition}
Given a shift space $X$, the \textbf{extender entropy of $X$} is given by $$h_E(X) = \lim_{n \rightarrow \infty} \frac{1}{n} \log{|E_X(n)|}.$$ The existence of this limit is given by Theorem~\ref{extentexists}.
\end{definition}

\begin{definition}
Given a shift space $X$, the \textbf{follower entropy of $X$} is given by $$h_F(X) = \limsup_{n \rightarrow \infty} \frac{1}{n} \log{|F_X(n)|}.$$ The \textbf{predecessor entropy} $h_P(X)$ is defined in an analogous fashion using predecessor sets. %This paper will show that follower entropy is not a conjugacy invariant, though having follower entropy equal to zero is a conjugacy invariant.
\end{definition}

By (\ref{proj}), for any $X$, $h_E(X) \geq h_P(X)$ and $h_E(X) \geq h_F(X)$.

\begin{remark}
It is clear from the upper bounds observed in Example \ref{CFS} that for the context-free shift $C$, we have $h_F(C) = h_P(C) = h_E(C) = 0$.
\end{remark}

\begin{definition}
For any two-sided shift space $X$, define $\widehat{X}$ to be its \textbf{reversed shift}, i.e. $\ldots x(-1) x(0) x(1) \ldots \in \widehat{X}$ iff $\ldots x(1) x(0) x(-1) \ldots \in X$. 
\end{definition}

The following lemmas are immediate from definitions and are presented without proof.

\begin{lemma}\label{reverse}
For any two-sided shift space $X$, $(X, \sigma^{-1})$ and $(\widehat{X}, \sigma)$ are conjugate.
\end{lemma}

\begin{lemma}\label{switch}
For any two-sided shift space $X$, $h_F(X) = h_P(\widehat{X})$ and $h_P(X) = h_F(\widehat{X})$.
\end{lemma}

\begin{example}
Given $\beta > 1$, let $d_\beta:[0,1) \rightarrow \{0, \ldots, \lceil \beta \rceil - 1\}^\mathbb{N}$ be the map which sends each point $x \in [0,1)$ to its expansion in base $\beta$. That is, if $\displaystyle x = \sum_{n=1}^\infty \frac{x_n}{\beta^n}$, then $d_\beta(x) = .x_1x_2x_3...$. (In the case where $x$ has more than one $\beta$-expansion, we take the lexicographically largest expansion.) The closure of the image, $\overline{d_\beta([0,1))}$, is a one-sided shift space called the \textbf{$\beta$-shift}, denoted $X_\beta$ and originally defined in \cite{renyi}.  %Therefore the $\beta$-shift must have alphabet $\{0, 1, ..., \lfloor \beta \rfloor \}$. 
An equivalent characterization of the $\beta$-shift is given by the right-infinite sequence $d_\beta^* (1) := \displaystyle \lim_{x \nearrow 1} d_\beta(x)$. For any sequence $x$ on the alphabet $\{0, ..., \lfloor \beta \rfloor \}$, $x \in X_\beta$ if and only if every $\sigma^n x$ is lexicographically less than or equal to $d_\beta^* (1)$ (see \cite{parry}). 

Though $\beta$-shifts are one-sided shift spaces as defined, each $X_\beta$ has a two-sided version defined as the set of all sequences on the same alphabet where every subword is a subword of some point in the one-sided version. (This is just the so-called natural extension.) Unless otherwise stated, $\beta$-shifts in this paper refer to those two-sided versions.

%Then clearly, the sequence $d_\beta^* (1)$ has the property that every $\sigma^n d_\beta^*(1)$ is lexicographically less than or equal to $d_\beta^* (1)$. %The sequence $d_\beta(1)$ only terminates with an infinite string of 0's in the case that $X_\beta$ is a shift of finite type, however, the sequence $d_\beta^*(1)$ never terminates with an infinite string of 0's. (Moreover, $d_\beta^*(1) \neq d_\beta(1)$ if and only if $X_\beta$ is a shift of finite type (\cite{blanchard})). By the results in \cite{french2}, for any non-sofic $\beta$-shift, $|F_{X_\beta}(n)| = n+1$, $|P_{X_\beta}(n)| = |L_n(X_\beta)|$, and $|E_{X_\beta}(n)| \leq |F_{X_\beta}(n)| \cdot |P_{X_\beta}(n)|$. Thus, the follower entropy of a $\beta$-shift must be 0, while the predecessor and extender entropies must be equal, as $|P_{X_\beta}(n)| \leq |E_{X_\beta}(n)| \leq |P_{X_\beta}(n)| \cdot (n+1)$.
\end{example}

\begin{remark}
There is yet another equivalent definition of $\beta$-shifts. If $T_\beta: [0,1) \rightarrow [0,1)$ is given by $T_\beta(x) = \beta x \pmod 1$, then the one-sided $\beta$-shift $X_\beta$ is a symbolic coding of $T_\beta$: for any $i \in \mathbb{N}$, $x \in [0,1)$, $\sigma^i(d_\beta(x)) = k$ if and only if $T_\beta^i(x) \in [\frac{k}{\beta}, \frac{k+1}{\beta})$. So, $X_\beta$ is the smallest shift space containing such codings of all $x \in [0,1)$. Though this definition has great historical importance in the study of $\beta$-shifts, we will not use it further in this work.
\end{remark}

The following results about follower/predecessor set sequences for $\beta$-shifts were proved in \cite{french2}.

\begin{theorem}\label{thomas1}
For any $\beta$-shift $X_\beta$ and any $n$, $|F_{X_{\beta}}(n)| \leq n + 1$.
\end{theorem}

\begin{theorem}\label{thomas2}
For any $\beta$-shift $X_\beta$ and any $n$, $|P_{X_{\beta}}(n)|$ is equal to the number of $n$-letter subwords of $d^*_\beta(1)$.
\end{theorem}

The latter is particularly useful when combined with the following result of Schmeling.

\begin{theorem}[\cite{schmeling}, Theorems B and E]
For all $\beta$ outside a meager set of Lebesgue measure $0$, all words in $L(X_\beta)$ appear as subwords of $d^*_\beta(1)$.
\end{theorem}

Since $h(X_\beta) = \log \beta$ for all $\beta$, the following corollary is immediate and will be useful in several later arguments.

\begin{corollary}\label{betacor}
For all $\beta$, $h_F(X_\beta) = 0$. For all $\beta$ outside a meager set of Lebesgue measure $0$, $h_E(X_\beta) = h_P(X_\beta) = h(X_\beta) = \log \beta$.

\end{corollary}

Finally, we briefly recall some definitions from \cite{buzzi} which are somewhat similar to ours.

\begin{definition}
Given a shift space $X$, a word $w \in L(X)$ is a \textbf{left constraint} if we can write $w = av$ for $a \in A$ and $F_X(av) \neq F_X(v)$. 
\end{definition}

\begin{definition}
Given a shift space $X$, the \textbf{left constraint entropy} is 
\[
h_C(X) = \limsup_{n \rightarrow \infty} \frac{\log |\{w \in L_n(X) \ : \ w \textrm{ is a left constraint}\}|}{n}.
\]
\end{definition}

Buzzi also defined right constraints and right constraint entropy, and proved many useful structural properties on $X$ under the hypothesis that either constraint entropy is strictly smaller than the topological entropy; such shift spaces were called \textbf{subshifts of quasi-finite type} in \cite{buzzi}. Any $\beta$-shift, for instance, is a subshift of quasi-finite type.

More directly relevant to our definitions is the fact that left constraint entropy is a conjugacy invariant which may increase under factor maps, just like extender entropy. One interesting observation is that the definition of left constraint entropy involves only follower sets, and our follower entropy is not a conjugacy invariant. This can be explained by another difference; in the definition of left constraint entropy, a single follower set corresponding to multiple words is counted multiple times, whereas in our definitions it would be counted only once. 

%ADD EXPLICIT RELATIONSHIPS BETWEEN OUR ENTROPIES; THERE'S AN INEQUALITY I THINK...

\section{The behavior of extender and follower entropy under products, conjugacies, and factors}
\label{ExtEnt}

\begin{theorem}\label{extentexists}
The limit in the definition of extender entropy exists for every shift space $X$.
\end{theorem}

\begin{proof}

We claim that for every $m,n \in \mathbb{N}$, $|E_X(n+m)| \leq |E_X(n)| \cdot |E_X(m)|$. Once this is verified, the existence of the limit follows from usual submultiplicativity arguments (i.e. the application of Fekete's Lemma.) 

To this end, we will define an injection $g: E_X(n+m) \rightarrow E_X(n) \times E_X(m)$. The definition of $g$ is as follows: for each extender set $E$ in $E_X(n+m)$, arbitrarily choose a word $w \in L_{n+m}(X)$ for which $E = E_X(w)$, then write $w = uv$ for $u \in L_n(X)$ and $v \in L_m(X)$, and define $g(E) = (E_X(u), E_X(v))$. 

It's clear that $g$ has the claimed domain and co-domain. We must only check that it is injective. To see this, we assume that $g(E) = g(E') = (F_1, F_2)$ for $E,E' \in E_{X}(n+m)$. Then $E = E_X(u_1 u_2)$ and $E' = E_X(u'_1 u'_2)$ where $E_X(u_1) = E_X(u'_1) = F_1$ and $E_X(u_2) = E_X(u'_2) = F_2$. Then, for any $(\ell, r) \in E$, $\ell u_1 u_2 r \in X$. Since $E_X(u_1) = E_X(u'_1)$, $\ell u'_1 u_2 r \in X$. Since $E_X(u_2) = E_X(u'_2)$, $\ell u'_1 u'_2 r \in X$, and so $(\ell, r) \in E_X(u'_1 u'_2)$. Since $(\ell, r)$ was arbitrary, $E = E_X(u_1 u_2) \subset E_X(u'_1 u'_2) = E'$. A trivially similar proof shows that $E' \subset E$, so $E = E'$, and we've verified injectivity of $g$. This implies the claimed inequality $|E_{X}(n+m)| \leq |E_X(n)| \cdot |E_X(m)|$, and so the existence of extender entropy as outlined above.

\end{proof}

\begin{theorem}\label{extprod}
Extender entropy is additive under products, i.e. if $X_1, X_2$ are shift spaces, then $h_E(X_1 \times X_2) = h_E(X_1) + h_E(X_2)$.
\end{theorem}

\begin{proof}
Any word in $L_n(X_1 \times X_2)$ can be written as a pair $(w_1, w_2)$, where $w_i$ is given by the $i$th coordinate of letters in $w$. It is immediate from the definition of Cartesian product and extender set that $E_{X_1 \times X_2}(w_1, w_2) = \{((s, s'), (u, u')) \ : \ (s, u) \in E_{X_1}(w_1), (s', u') \in E_{X_2}(w_2)\}$. Then, $(w_1, w_2)$ and $(w'_1, w'_2)$ have the same extender set in $X_1 \times X_2$ iff $w_1$ and $w'_1$ have the same extender set in $X_1$ and $w_2$ and $w'_2$ have the same extender set in $X_2$. Therefore
\[
|E_{X_1 \times X_2}(n)| = |E_{X_1}(n)| |E_{X_2}(n)|,
\]
and taking logarithms, dividing by $n$, and letting $n \rightarrow \infty$ yields $h_E(X_1 \times X_2) = h_E(X_1) + h_E(X_2)$.

\end{proof}

\begin{theorem}\label{extconjinv}
Extender entropy is a conjugacy invariant, i.e. if $X$ and $Y$ are conjugate shift spaces, then $h_E(X) = h_E(Y)$.
\end{theorem}

\begin{proof}
Suppose that $X$ and $Y$ are conjugate shift spaces, via a conjugacy $\phi: X \rightarrow Y$. Since $\phi$ is a sliding block code, it has a radius $r$, meaning that $x(-r) \ldots x(r)$ uniquely determines $(\phi(x))(0)$. We can then decompose $\phi = \psi \circ f^{[-r,r]}$, where 
$f^{[-r,r]}: X \rightarrow X^{[-r,r]}$ is the canonical conjugacy from $X$ to its higher-block presentation $X^{[-r,r]}$ and $\psi: X^{[-r,r]} \rightarrow Y$ is a $1$-block conjugacy. We will prove Theorem~\ref{extconjinv} by showing first that $h_E(X) = h_E(X^{[-r,r]})$ and then that $h_E(X^{[-r,r]}) \geq h_E(Y)$, implying that $h_E(X) \geq h_E(Y)$; reversing the roles of $X$ and $Y$ then shows that 
$h_E(X) = h_E(Y)$. For notational convenience, we from now on use $Z$ to denote $X^{[-r,r]}$ and $f$ to denote $f^{[-r,r]}$. We need the following auxiliary fact:\\

\noindent
\textbf{Fact 1:} For $n \geq 2r$ and $w, w' \in L_n(X)$, $E_Z(f(w)) = E_Z(f(w'))$ if and only if $E_X(w) = E_X(w')$ and $w(i) = w'(i)$ for all $1 \leq i \leq 2r$ and $n - 2r < i \leq n$.\\

$\Longrightarrow$: Assume that $E_Z(f(w)) = E_Z(f(w'))$. Note that by the rules defining 
$Z = X^{[-r,r]}$, any letter $a \in A^{[-r,r]}$ (the alphabet of $Z$) for which $f(w) a \in L(Z)$ must have first $2r$ letters (in $A$) equal to the final $2r$ letters (in $A$) of $w$. Therefore, it must be the case that $w$ and $w'$ agree on their final $2r$ letters, and a trivially similar argument shows the same about the first $2r$ letters. 

Now, choose any $(u,v) \in E_X(w)$. Then $uwv \in X$, and so clearly $f(uwv) \in Z$. We can then write $f(uwv) = f(u) s f(w) t f(v)$, where $s$ is determined by the final $2r$ letters of $u$ and initial $2r$ letters of $w$, and $t$ is determined by the final $2r$ letters of $w$ and initial $2r$ letters of $v$. Since $w,w'$ agree on their first and last $2r$ letters, $f(uw'v) = f(u) s f(w') t f(v)$. But then since $E_Z(f(w)) = E_Z(f(w'))$, and $f(uwv) = f(u) s f(w) t f(v) \in Z$, it must be the case that $f(uw'v) = f(u) s f(w') t f(v) \in Z$. But then since $f$ is invertible, $uw'v \in X$ and so $(u,v) \in E_X(w')$. Since $(u,v)$ was an arbitrary element of $E_X(w)$, we've shown that $E_X(w) \subseteq E_X(w')$, and a trivially similar argument shows the reverse, so $E_X(w) = E_X(w')$.\\

$\Longleftarrow$: Assume that $E_X(w) = E_X(w')$ and that $w(i) = w'(i)$ for all $1 \leq i \leq 2r$ and $n - 2r < i \leq n$. Choose any $(u,v) \in E_Z(f(w))$, meaning that $uf(w)v \in Z$. Then clearly $f^{-1}(u f(w) v) \in X$, and can be written as $\overline{u} w \overline{v}$. Similar arguments to those used in the reverse direction show that since $w,w'$ share the same first and last $2r$ letters, $f^{-1}(u f(w') v) = \overline{u} w' \overline{v}$. Since $E_X(w) = E_X(w')$ and $f^{-1}(u f(w) v) = \overline{u} w \overline{v} \in X$, it must be the case that $f^{-1}(u f(w') v) = \overline{u} w' \overline{v} \in X$ as well. But then $u f(w') v \in Z$. Since $u,v$ were arbitrary, we've shown that $E_Z(f(w)) \subseteq E_Z(f(w'))$, and a trivially similar argument shows the reverse, so $E_Z(f(w)) = E_Z(f(w'))$, completing the proof of Fact 1.\\

Fact 1 implies that for every $n > 4r$, $|E_X(n)| \leq |E_Z(n)| \leq |E_X(n)| \cdot |A|^{4r}$. Upon taking logarithms, dividing by $n$, and letting $n \rightarrow \infty$, we see that $h_E(X) = h_E(Z)$. It remains to prove that $h_E(Z) \geq h_E(Y)$, for which we will use only the fact that $\psi: Z \rightarrow Y$ is a $1$-block conjugacy. Since $\psi^{-1}$ is a sliding block code, it has a radius $s$, meaning that $(\psi(x))(-s) \ldots (\psi(x))(s)$ uniquely determines $x(0)$. The final step in our proof is the following auxiliary fact:\\

\noindent
\textbf{Fact 2:} If $E_Z(w) = E_Z(w')$ and $u, v \in A^s$ are such that $uwv \in L(Z)$ (meaning that $uw'v \in L(Z)$ also), then $E_Y(\psi(w)) = E_Y(\psi(w'))$.\\

To see this, suppose that $E_Z(w) = E_Z(w')$ and $u, v \in A^s$ are such that $uwv, uw'v \in L(Z)$, and recall that $\psi$ is a $1$-block map. Our key observation is that

\begin{equation}\label{eq1}
E_Y(\psi(uwv)) = \bigcup_{x \in \psi^{-1}(\psi(uwv))} \psi(E_Z(x)) = \bigcup_{y \in \psi^{-1}(u), z \in \psi^{-1}(v) \textrm{ s.t. }  ywz \in L(Z)} \psi(E_Z(ywz)).
\end{equation}
Here, the second equality uses the fact that any $x \in \psi^{-1}(uwv)$ must have a $w$ at its center since $s$ is the radius of $\psi^{-1}$. 

We now note that since $E_Z(w) = E_Z(w')$, the sets $\{y \in \psi^{-1}(u), z \in \psi^{-1}(v) \textrm{ s.t. } \newline ywz \in L(Z)\}$ and $\{y \in \psi^{-1}(u), z \in \psi^{-1}(v) \textrm{ s.t. } yw'z \in L(Z)\}$ are the same. For the same reason, given any pair $y,z$ in this set, 
$E_Z(ywz) = E_Z(yw'z)$. Combining with (\ref{eq1}), we see that indeed $E_Y(\psi(uwv)) = E_Y(\psi(uw'v))$, completing the proof.\\

By Fact 2, for every $n > 2s$, $|E_Y(n)| \leq |E_{Z}(n-2s)| \cdot |A|^{2s}$. Upon taking logarithms, dividing by 
$n$, and letting $n \rightarrow \infty$, we see that $h_E(Z) \geq h_E(Y)$. Combining with the fact that $h_E(X) = h_E(Z)$, we see that
$h_E(X) \geq h_E(Y)$, and since we could repeat this proof with the roles of $X$ and $Y$ reversed, we see that indeed $h_E(X) = h_E(Y)$, completing the proof of Theorem~\ref{extconjinv}. 

\end{proof}

Unfortunately, we will show that follower entropy is not conjugacy-invariant. However, the following result bounds the amount by which it can increase under a conjugacy.

\begin{theorem}\label{follconjbound}
If $\phi: X \rightarrow Y$ is a conjugacy, $r$ is the radius of $\phi$, and $s$ is the radius of $\phi^{-1}$, then $h_F(Y) \leq 2(r+s)h_F(X)$.
\end{theorem}

\begin{proof}
Suppose that $\phi: X \rightarrow Y$ is a conjugacy with radius $r$ and whose inverse has radius $s$. As before, we decompose 
$\phi = \psi \circ f^{[-r,r]}$, where $f^{[-r,r]}: X \rightarrow X^{[-r,r]}$. We will prove Theorem~\ref{follconjbound} by showing first that $h_F(X) = h_F(X^{[-r,r]})$ and then that $h_F(Y) \leq 2(r+s) h_F(X^{[-r,r]})$. We again use $Z$ to denote $X^{[-r,r]}$ and $f$ to denote $f^{[-r,r]}$. We need the following auxiliary fact: for any $n \geq 2r$ and $w,w' \in L_n(Z)$, $F_Z(f(w)) = F_Z(f(w'))$ if and only if $F_X(w) = F_X(w')$ and $w(i) = w'(i)$ for $n - 2r < i \leq n$. We omit the proof, as it is trivially similar to that of Fact 1 in the proof of Theorem~\ref{extconjinv}.

Then, for every $n > 2r$, $|F_X(n)| \leq |F_Z(n)| \leq |F_X(n)| \cdot |A|^{2r}$. Upon taking logarithms, dividing by $n$, and taking the limsup as $n \rightarrow \infty$, we see that $h_F(X) = h_F(Z)$. It now suffices to show that 
$h_F(Y) \leq 2(r+s) h_F(Z)$. We note that $\psi: Z \rightarrow Y$ is a $1$-block map, and that its inverse $\psi^{-1}$ has radius $r+s$. 
Our key observation is that for any $w \in L(Y)$ and $u,v \in A^{r+s}$ for which $uwv \in L(Y)$,

\begin{equation}\label{eq2}
F_Y(\psi(uwv)) = \bigcup_{x \in \psi^{-1}(\psi(uwv))} \psi(F_Z(x)) = \bigcup_{y \in \psi^{-1}(u), z \in \psi^{-1}(v) \textrm{ s.t. } ywz \in L(Z)} \psi(F_Z(ywz)).
\end{equation}

This means that for $n > 2(r+s)$, each follower set of an $n$-letter word in $Y$ is determined by a collection of at most $|A|^{2(r+s)}$ follower sets of $n$-letter words in $Z$. Therefore,

\[
|F_Y(n)| \leq \sum_{i = 1}^{2(r+s)} {|F_X(n)| \choose i} \leq \sum_{i = 1}^{2(r+s)} |F_X(n)|^i \leq 2(r+s) |F_X(n)|^{2(r+s)}.
\]

Taking logarithms, dividing by $n$, and taking the limsup as $n \rightarrow \infty$ yields $h_F(Y) \leq 2(r+s) h_F(Z) = 2(r+s) h_F(X)$, completing the proof of Theorem~\ref{follconjbound}.

\end{proof}

As an immediate corollary, we see that having zero follower entropy is a conjugacy-invariant condition.

\begin{corollary}\label{0follentinv}
Having zero follower entropy is a conjugacy invariant, i.e. if $X$ and $Y$ are conjugate shift spaces and $h_F(X) = 0$, then $h_F(Y) = 0$.
\end{corollary}

Follower entropy itself is, however, not conjugacy invariant.

\begin{theorem}\label{follnotconjinv}
There exist a pair of conjugate shift spaces $X,Y$ with $h_F(X) \neq h_F(Y)$.
\end{theorem}

\begin{proof}
We begin by defining a shift space $X_1$ with alphabet $\{0,1\}$ and $h_F(X_1) > 0$. For instance, by Corollary~\ref{betacor}, we can choose $X_\beta$ with $1 < \beta < 2$ satisfying $h_P(X_\beta) = \log \beta$ and define $X_1 = \widehat{X_\beta}$. Then by Lemma~\ref{switch}, $h_F(X_1) = h_P(X_\beta) = \log \beta > 0$. Then, define $X_2$ with alphabet $\{0,1,0',1'\}$ to be the disjoint union $X_1 \cup X'_1$, where $X_1'$ consists of points of $X_1$ with primes placed above every letter. Clearly $|F_{X_2}(n)| = 2|F_{X_1}(n)|$, and so $h_F(X_2) = h_F(X_1)$. 

Now, we will define $X_3$ with alphabet $\{0,1,0',1',a,b,c\}$. Firstly, letters from $\{a,b,c\}$ cannot appear consecutively or at distance $1$ in points of $X_3$, i.e. we declare all words $xy$ and $xzy$ for $x,y \in \{a,b,c\}$ and any $z$ to be forbidden. Secondly, for any word of the form
\[
w = w_0 d_1 w_1 d_2 w_2 \ldots d_n w_n,
\]
with $d_i \in \{a,b,c\}$ and $w_i \in \{0,1,0',1'\}^*$ ($w_0$ may be the empty word), define an auxiliary word $\eta(w)$ in $\{0,1,0',1'\}^*$ as follows: whenever $d_i = a$, use the odd-indexed letters from $w_i$, whenever $d_i = b$, use the even-indexed letters from $w_i$, and whenever $d_i = c$, ignore $w_i$ entirely. Then, $\eta(w)$ is formed by concatenating the corresponding words for each $i$. As an example, $\eta(1'a011'0'b110'c11'1a0'1) = 01'10'$. We say that $w \in L(X_3)$ iff $\eta(w) \in L(X_2)$. Note that this includes some degenerate cases, for instance $\{0,1,0',1'\}^* \subset L(X_3)$ because for any $w \in \{0,1,0',1'\}^*$, $\eta(w)$ is the empty word, which is vacuously in $L(X_2)$.

We note that for any $w \in L_n(X_3)$, $F_{X_3}(w)$ is determined entirely by its rightmost letter from $\{a,b,c\}$ and by $\eta(w) \in L(X_2)$ of length at most $n/2$. Therefore, 
$|F_{X_3}(n)| \leq 3 \sum_{i = 0}^{\lfloor n/2 \rfloor} |F_{X_2}(i)|$, and taking logarithms, dividing by $n$, and taking the limsup as $n \rightarrow \infty$ shows that $h_F(X_3) \leq \frac{1}{2} h_F(X_2) = \frac{1}{2} h_F(X_1)$.

We now define $X_4$ to be the higher-block recoding $X_3^{[0,1]}$. By the proof of Theorem~\ref{follconjbound}, $h_F(X_4) = h_F(X_3) \leq \frac{1}{2} h_F(X_1)$. The alphabet of $X_4$ is now $L_2(X_3) \subseteq \{0,1,0',1',a,b,c\}^2$. Finally, we define a $1$-block map $\psi$ on $X_4$ as follows: all letters of the alphabet of $X_4$ of the form $az$ or $bz$ are sent under $\psi$ to a new symbol $*$, and on all other letters of the alphabet of $X_4$, $\psi$ acts as the identity. Define $X_5 = \psi(X_4)$. We claim that $\psi$ is injective, and therefore that $X_4$ and $X_5$ are conjugate. To see this, simply note that for any $x \in X_4$, $(\psi(x))(-1) (\psi(x))(0) (\psi(x))(1)$ uniquely determines $x(0)$; if $(\psi(x))(0) \neq *$, then $(\psi(x))(0) = x(0)$, and if $(\psi(x))(0) = *$, then since $a,b$ cannot appear consecutively or separated by distance $1$ in $X_3$, both $(\psi(x))(-1)$ and $(\psi(x))(1)$ are not $*$. Therefore,
$(\psi(x))(-1) = x(-1)$ and $(\psi(x))(1) = x(1)$, and by the definition of $X_4$ as a higher-block recoding, $x(-1)$ and $x(1)$ determine $x(0)$. We have then shown that $X_4$ and $X_5$ are conjugate, and it remains only to show that $h_F(X_5) > h_F(X_4)$. 

To see this, we begin by defining an injection from $F_{X_1}(n) \times F_{X'_1}(n)$ to $F_{X_5}(2n)$. For any pair $(F, F') \in F_{X_1}(n) \times F_{X'_1}(n)$, arbitrarily choose $v \in L_n(X_1)$ and $v' \in L_n(X'_1)$ for which  $F = F_{X_1}(v)$ and $F' = F_{X'_1}(v')$. Then, define a word $\pi(F, F')$ by 
\[
\pi(F, F') = * (v(1) v'(1)) (v'(1) v(2)) (v(2) v'(2)) \ldots (v(n) v'(n)).
\]

Firstly, we claim that $\pi(F, F') \in L_{2n}(X_5)$; this is because $v \in L(X_1)$, and so 
$a v(1) v'(1) \ldots v(n) v'(n) \in L(X_3)$. Next, we claim that for any pairs $(F_1, F'_1) \neq (F_2, F'_2) \in F_{X_1}(n) \times F_{X'_1}(n)$, $F_{X_5}(\pi(F_1, F'_1)) \neq F_{X_5}(\pi(F_2, F'_2))$. Either $F_1 \neq F_2$ or $F'_1 \neq F'_2$; we begin with the former case. Let's say that $v_i \in L_n(X_1)$ and $v'_i \in L_n(X'_1)$ ($i = 1,2$) were used in the definition of $\pi(F_i, F'_i)$, meaning that $F_i = F_{X_1}(v_i)$ and $F'_i = F_{X'_1}(v'_i)$. Without loss of generality, we assume that there is a word $w = w(1) \ldots w(m) \in F_1 \setminus F_2$.

Now, we claim that 
\[
u = (v_1(n) 0) (0 w(1)) (w(1) 0) (0 w(2)) (w(2) 0) \ldots (0 w(m)) (w(m) 0)
\]
is in $F_{X_5}(\pi(F_1, F'_1))$, but is not in $F_{X_5}(\pi(F_2, F'_2))$. To see that $\pi(F_1, F'_1) u \in L(X_5)$, we just note that 
$v_1 w \in L(X_1)$, and so 
\[
a v_1(1) v'_1(1) v_1(2) v'_1(2) \ldots v_1(n) v'_1(n) w(1) 0 w(2) 0 \ldots w(m) 0 \in L(X_3).
\] 

Assume for a contradiction that $\pi(F_2, F'_2) u \in L(X_5)$. Then, by definition of $\psi$ and $X_4$ as a higher-block presentation, there must be some word of the form
\[
t v_2(1) v'_2(1) \ldots v_2(n) v'_2(n) w(1) 0 w(2) 0 \ldots w(m) 0 
\]
in $L(X_3)$, where $t \in \{a,b\}$. However, clearly $t \neq b$, since if it were, the word $v'_2 0^m$ would have to be in $L(X_2)$, which is impossible since the letters of $v'_2$ are in $\{0',1'\}$. Therefore, $t = a$, and so by definition of $X_3$, $v_2 w \in L(X_1)$, a contradiction to $w \notin F_2 = F_{X_1}(v_2)$. This implies that $F_{X_5}(\pi(F_1, F'_1)) \neq F_{X_5}(\pi(F_2, F'_2))$ when $F_1 \neq F_2$. The case for $F'_1 \neq F'_2$ is trivially similar, and so the map sending $(F, F')$ to $F_{X_5}(\pi(F, F'))$ is injective.

This implies that $|F_{X_5}(2n)| \geq |F_{X_1}(n)| \cdot |F_{X'_1}(n)| = |F_{X_1}(n)|^2$. Taking logarithms, dividing by $2n$, and taking the limsup as $n \rightarrow \infty$ shows that $h_F(X_5) \geq h_F(X_1)$. Since we showed earlier that $h_F(X_4) \leq \frac{1}{2} h_F(X_1)$, $h_F(X_5) > h_F(X_4)$, and the theorem is proved.

\end{proof}

We will show that under factor maps, nothing similar can be said, and also that even though extender entropy is a conjugacy invariant, it does not necessarily decrease under factor maps.

\begin{theorem}\label{entfact}
There exist a shift space $X$ and a factor map $\phi: X \rightarrow Y$ such that $h_E(X) = h_F(X) = 0$, but $h_E(Y), h_F(Y) > 0$.
\end{theorem}

\begin{proof}
The alphabet of $X$ will be $\{1,2,3,a,b,c\}$. The points of $X$ are exactly those sequences
\[
\ldots a_0 x^{(1)}_0 x^{(2)}_0 x^{(3)}_0 a_1 x^{(1)}_1 x^{(2)}_1 x^{(3)}_1 a_2 \ldots
\]
with $a_n \in \{1,2,3\}$ for every $n \in \mathbb{Z}$, $x^{(j)}_n \in \{a,b,c\}$ for all $n \in \mathbb{Z}$ and $j \in \{1,2,3\}$, and for which the point $(x^{(a_n)}_n)_{n \in \mathbb{Z}}$ is in the context-free shift $C$.

First, we claim that $h_E(X) = 0$, which will trivially imply that $h_F(X) = 0$. To see this, we first claim that for $v \in L(X)$ of the form $a_1 x^{(1)}_1 x^{(2)}_1 x^{(3)}_1 \ldots a_n x^{(1)}_n x^{(2)}_n x^{(3)}_n$, $E_X(v)$ is completely determined by $E_C(\kappa(v))$, where $\kappa(v)$ denotes $x^{(a_1)}_1 \ldots x^{(a_n)}_n \in L(C)$; the proof is left to the reader. Then, for an arbitrary word $w \in L(X)$, we can decompose as $w = pvs$, where $v$ has the form above, $p$ has length at most $3$ and consists of letters in $\{a,b,c\}$, and $s$ has length at most $3$, has first letter in $\{1,2,3\}$, and all remaining letters in $\{a,b,c\}$. Then, $E_X(w)$ is a projection of $E_X(v)$ determined by $p$ and $s$, meaning that $E_X(w)$ is completely determined by $E_C(\kappa(v))$, $p$, and $s$. Therefore,
\[
|E_X(n)| \leq (27 + 9 + 3 + 1)^2 E_C(\lceil n/4 \rceil).
\]
Since $h_E(C) = 0$, taking logarithms, dividing by $n$, and letting $n \rightarrow \infty$ shows that $h_E(X) = 0$, trivially implying that $h_F(X) = 0$ as well.

We now define a simple $1$-block factor map on $X$: $\phi$ maps $1$, $2$, and $3$ to a new symbol $*$ and each of $a,b,c$ to themselves. We denote $\phi(X)$ by $Y$, a shift space with alphabet $\{*,a,b,c\}$. We claim that $h_F(Y) > 0$, which will obviously imply that $h_E(Y) > 0$ as well.
To verify this, we will show that for any $n$, all words of the form $* w_1 * w_2 \ldots * w_n$, where each $w_i$ is either $aba$ or 
$abc$, are in $L(Y)$, and have distinct follower sets in $Y$.

The first claim is fairly simple: clearly, since $a^{\infty} \in C$, for any choice of the words $w_n$ (between $aba$ and $abc$), the point 
$\ldots 1 w_1 1 w_2 1 w_3 \ldots$ is in $X$, and so every point of the form $\ldots * w_1 * w_2 * w_3 \ldots$ is in $Y$. To see the second claim, consider any $n$ and unequal words $v = * w_1 * w_2 \ldots * w_n$ and $v' = * w'_1 * w'_2 \ldots * w'_n$. Since $v \neq v'$, there exists $i$ so that $w_i \neq w'_i$; without loss of generality we assume that $w_i = abc$ and $w'_i = aba$.

We claim that the word $u = (*aaa)^i (*bbb)^n *ccc$ is in $F_Y(v) \setminus F_Y(v')$. To see that $vu \in L(Y)$ (i.e. that $u \in F_Y(v)$), we simply note that 
\[
2 w_1 2 w_2 \ldots 2 w_{i-1} 3 w_i 1 w_{i+1} \ldots 1 w_n (1 aaa)^i (1 bbb)^n 1 ccc
\]
is in $L(X)$, since the word spelled out by the indicated letters is $b^{i-1} c a^n b^n c$, which is in $L(C)$. The $\phi$-image of this word is $vu$, which is therefore in $L(Y)$. 

Assume for a contradiction that $v'u \in L(Y)$. Then, some word of the form
\[
a_1 w'_1 a_2 w'_2 \ldots a_n w'_n a_{n+1} aaa \ldots a_{n+i} aaa a_{n+i+1} bbb \ldots a_{2n+i} bbb a_{2n+i+1} ccc
\]
would have to be in $L(X)$, where each $a_i \in \{1,2,3\}$. In particular, this implies that a word of the form
\[
b_1 b_2 \ldots b_n a^i b^n c
\]
is in $L(C)$, where each $b_i \in \{a,b,c\}$. Note that since $w'_i = aba$, $b_i \neq c$. However, by definition of $C$, $a^i b^n c$ for $n > i$ must be immediately preceded by $c a^{n-i}$ in any point of $C$. This is a contradiction since $b_i \neq c$. Therefore, $v'u \notin L(Y)$. We have then shown that $F_Y(v) \neq F_Y(v')$, and since $v$ and $v'$ were arbitrary, that $|F_{Y}(4n)| \geq 2^n$. We take logarithms, divide by $n$, and take the limsup as $n \rightarrow \infty$ to see that $h_F(Y) \geq \frac{\log 2}{4} > 0$; clearly $h_E(Y) \geq h_F(Y) > 0$, and our proof is complete.

\end{proof}

\section{Applications}
\label{achieve}

The key to both applications mentioned in the introduction is Corollary~\ref{betacor}.

\begin{theorem}\label{notinv}
For all $\beta$ outside a meager set of zero Lebesgue measure, the $\beta$-shift $(X_{\beta}, \sigma)$ is not conjugate to its inverse. 
\end{theorem}

\begin{proof}
For any $\beta$ outside the meager set of Lebesgue measure $0$ from Corollary~\ref{betacor}, $h_F(\beta) = 0$ and $h_P(X_\beta) > 0$. Then by Lemma~\ref{switch}, $h_F(\widehat{X_{\beta}}) = h_P(X_{\beta}) > 0$, and so by Corollary~\ref{0follentinv}, $(X_\beta, \sigma)$ and $(\widehat{X_{\beta}}, \sigma)$ are not conjugate. Then by Lemma~\ref{reverse}, $(X_\beta, \sigma)$ and $(X_{\beta}, \sigma^{-1})$ are also not conjugate.
\end{proof}

\begin{remark}
This fact could be proved in a similar fashion by using the left and right constraint entropies from \cite{buzzi}.
\end{remark}

%We may now use Corollary~\ref{betacor} to prove Theorem~\ref{realization}, i.e. that for any $x \leq y$, there exists a shift space $X$ with $h_E(X) = x$ and $h(X) = y$.
%This will be sufficient to show that any pair of entropies $(h(X), h_E(X))$ with $h_E(X) \leq h(X)$ 

\begin{theorem}\label{realization}
For any $x \leq y$, there exists a shift space $X$ with $h_E(X) = x$ and $h(X) = y$.
\end{theorem}

\begin{proof}

It is clearly sufficient to treat the cases $x = 0$ and $x = y$ only, since $h(X)$ and $h_E(X)$ are additive under products (Theorem~\ref{extprod}), and one can write $(x, y)$ as $(0, y - x) + (x, x)$.

Consider any $x > 0$. By Corollary~\ref{betacor}, there exists $\alpha > x$ where $x/\alpha \notin \mathbb{Q}$ and where $h_E(X_\beta) = h(X_\beta) = \alpha$. Take $S$ to be a Sturmian shift space with rotation number $x/\alpha \in (0,1) \setminus \mathbb{Q}$. (We will not define Sturmian shift spaces here, but the only properties we will need is that for $S$ Sturmian with rotation number $\eta$, the alphabet is $\{0,1\}$ and every word in $L_n(S)$ contains either $\lfloor n\eta \rfloor$ or $\lceil n\eta \rceil$ $1$ symbols. See Chapter 6 of \cite{fogg} for a complete introduction to Sturmian shifts.)

Then, define a new shift space $Z$ as follows. Choose $* \notin A$ and define the alphabet of $Z$ to be $(*, 0) \sqcup (A \times \{1\})$. If we write a sequence $z$ on this alphabet as $z = ((z(n))_1, (z(n))_2)_{n \in \mathbb{Z}}$, then the shift space $Z$ is the set of all such $z$ where $((z(n))_2)_{n \in \mathbb{Z}} \in S$ and $((z(n))_1)_{\{n \ : \ z(n)_2 = 1\}} \in X_\beta$, i.e. the first coordinates of $z$ at locations where the second coordinate has a $1$, taken in order, comprise a point of $X_\beta$.

We claim that $h(Z) = h_E(Z) = x$. Let's first show that $h(Z) = x$. For any $v \in L_n(S)$, denote by $|v|_1$ the number of $1$ symbols in $v$. Then, for each $u \in L_{|v|_1}(X_\beta)$, we can create $w(u,v) \in L_n(Z)$ by placing $v$ in the second coordinate, $*$ in the first coordinate at each location where $v$ has a $0$, and placing $u$ in the first coordinate along the set of indices where $v$ has $1$s. By definition, $L_n(Z)$ is in fact the set of all such $w(u,v)$, and the map $(u,v) \mapsto w(u,v)$ is a bijection. To estimate $|L_n(Z)|$, we therefore need only to estimate the number of such pairs $(u,v)$.

Recall that since $S$ is Sturmian, $|v|_1$ is always either $\lfloor nx/\alpha \rfloor$ or $\lceil nx/\alpha \rceil$, and that $|L_n(S)| = n + 1$. Therefore,
\[
(n+1) |L_{\lfloor nx/\alpha \rfloor}(X_\beta)| \leq |L_n(Z)| \leq (n+1) |L_{\lceil nx/\alpha \rceil}(X_\beta)|.
\]
But then, taking logarithms, dividing by $n$, and letting $n \rightarrow \infty$ yields 
$h(Z) = (x/\alpha) h(X_\beta) = x$.

Then $h_E(Z) \leq h(Z) = x$, so we need only show that $h_E(Z) \geq x$. For this, we fix $v \in L_n(S)$, and claim that if $u,u' \in L_{|v|_1}(X_\beta)$ and $E_{X_\beta}(u) \neq E_{X_\beta}(u')$, then $E_Z(w(u,v)) \neq E_Z(w(u',v))$. To see this, assume that $E_{X_\beta}(u') \not\subset E_{X_\beta}(u)$, choose a pair $(s,t) \in E_{X_\beta}(u') \setminus E_{X_\beta}(u)$, and choose 
$(a,b) \in E_S(v)$ where $a$ has $|s|$ $1$s and $b$ has $|t|$ $1$s. Then, create $p$ with $a$ in the second coordinate, $*$ in the first coordinate at all locations of $0$s in $a$, and $s$ spelled out along the first coordinate at indices where $a$ has $1$ symbols. Similarly create $q$ from $b$ and $t$. It's immediate from the definition of $z$ that $(p,q) \in E_Z(w(u',v)) \setminus E_Z(w(u,v))$, and so $E_Z(w(u,v)) \neq E_Z(w(u',v))$. Then, all distinct extender sets among $u \in L_{|v|_1}(X_\beta)$ yield distinct extender sets in words of $L_n(Z)$, i.e.
\[
|E_Z(n)| \geq |E_{X_\beta}(|v|_1)|.
\]
Since $|v|_1 = \lfloor nx/\alpha \rfloor$ or $|v|_1 = \lceil nx/\alpha \rceil$, taking logarithms, dividing by $n$, and letting $n \rightarrow \infty$ yields $h_E(Z) \geq (x/\alpha) h_E(X_\beta) = x$, completing the proof that 
$h(Z) = h_E(Z) = x$.  

It remains only to construct $Y$ with $h_E(Y) = 0$ and $h(Y) = x$. First, choose $n$ with $\log n > x$ and $x/\log n \notin \mathbb{Q}$, and define $F$ to be the full shift on $n$ symbols. Then, $h(F) = \log n$ and $h_E(F) = 0$. Perform exactly the same procedure as above, where $F$ replaces $X_\beta$ and $S$ is chosen to have rotation number $x/\log n$. Then, exactly as above, $h_E(Y) = (x/\log n) h_E(F) = 0$, and $h(Y) = (x/\log n) h(F) = x$, completing the proof.

\end{proof}

%\section*{acknowledgements} 

\bibliographystyle{plain}
\bibliography{ExtenderEntropy}

\end{document}